\documentclass[a4paper,12pt]{article}

\usepackage[T1]{fontenc}
\usepackage{amsfonts}

\usepackage{amsmath, amsthm, amssymb, mathrsfs}

\newcommand{\nhm}{nonhomogeneous }

\newtheorem{thrm}{Theorem}
\newtheorem{fact}{Fact}

\newtheorem{defn}{Definition}
\newtheorem{lemma}{Lemma}
\newtheorem{prop}{Proposition}
\newtheorem{remark}{Remark}

\newcommand{\rE}{\mathrm E}

\begin{document}

\author{Michael Grabchak\footnote{Email address: mgrabcha@uncc.edu}\ \ and Isaac M.~Sonin\footnote{Email address: imsonin@uncc.edu} \and University of North Carolina Charlotte}
\title{A Zero-One Law for Markov Chains}
\date{\today}
\maketitle

\begin{abstract}
We prove an analog of the classical Zero-One Law for both homogeneous and nonhomogeneous Markov chains (MC). Its almost precise formulation is simple: given any event $A$ from the tail $\sigma$-algebra of MC $(Z_n)$, for large $n$, with probability near one, the trajectories of the MC are in states $i$, where $P(A|Z_n=i)$ is either near $0$ or near $1$. A similar statement holds for the entrance $\sigma$-algebra, when $n$ tends to $-\infty$. To formulate this second result, we give detailed results on the existence of nonhomogeneous Markov chains indexed by $\mathbb Z_-$ or $\mathbb Z$ in both the finite and countable cases. This extends a well-known result due to Kolmogorov. Further, in our discussion, we note an interesting dichotomy between two commonly used definitions of MCs.
\end{abstract}

\section{Introduction}

This paper addresses two problems in the study of nonhomogeneous Markov Chains (MC), where we understand homogeneous MCs to be an important special case. The first is a zero-one law for MCs and the second is an extension of a result, due to Kolmogorov, on the existence of MCs indexed by the integers and thus not having an initial starting point. 

The zero-one law for MCs was first formulated in Sonin (1991) \cite{Son91a}, but without detailed proofs and with some gaps. We remedy this by giving two detailed proofs. The first is simple, but does not illustrate the underlying mechanics. The second is constructive and helps to illustrate how past (observed) events can inform future (tail) events. In the case where the MC is indexed by the integers, we also formulate a zero-one law for the entrance $\sigma$-algebra, when $n$ tends to $-\infty$.

The second topic is related to a problem first posed by Kolmogorov in a short paper from 1936 \cite{Kolm1936}. While that paper is usually remembered for introducing the concept of a reversible MC, most of it deals with the following question: Given a sequence of stochastic matrices $(P_{n})_{n\in\mathbb N_-}$, does there exist a Markov chain indexed by $\mathbb Z_-$ with this as its sequence of transition matrices? When the number of states is finite and fixed over time, \cite{Kolm1936} answers in the affirmative. In this paper, we extend that result to the case where the cardinalities are finite, but may be changing and may approach infinity. We further show that, in the case where the cardinalities are countably infinite, the MC may not exist and we give a sufficient condition for when it does. These results help to explain when a MC indexed by the integers exists and thus when we can talk about the entrance $\sigma$-algebra. In our discussion, we note an interesting dichotomy between two commonly used definitions of MCs. The first is in terms of a sequence of random variables that satisfies the Markov property and the second is in terms of a sequence of transition matrices (or, in the homogeneous case, just one transition matrix).

The rest of this paper is organized as follows. In Section \ref{sec: zero one law} we formulate the zero-one law for the tail $\sigma$-algebra. In Section \ref{sec: entrance} we extend this to the entrance $\sigma$-algebra and discuss the above mentioned dichotomy. In Section \ref{sec: MC past} we discuss the existence of MCs on $\mathbb Z_-$ and $\mathbb Z$ and extend the results of \cite{Kolm1936}. Proofs are postponed to Section \ref{sec: proofs all}. We conclude in Section \ref{sec: conc} by discussing some directions for future work. A short historical note is given in Section \ref{sec: hist}.

Before proceeding, we introduce some notation. Let $\mathbb Z_+=\{0,1,2,\dots\}$, $\mathbb N=\{1,2,\dots\}$, $\mathbb Z_-=\{\dots,-2,-1,0\}$, and $\mathbb N_-=\{\dots,-2,-1\}$.  Let $\mathbb R^d$ be the space of $d$-dimensional {\em row} vectors. For $m\in\mathbb R^d$, we write $m(i)$ to denote the $i$th coordinate of $m$. Let $e_i^{(d)}\in\mathbb R^d$ be the $i$th row of the $d\times d$-dimensional identity matrix. When we talk about convergence of a sequence of finite or infinite matrices, we mean pointwise convergence of their coordinates. For a set $A$ we write $|A|$ to denote its cardinality and $I_A$ to denote the indicator function on $A$. We write $\bigtriangleup$ to denote the symmetric difference operator. If $E,E_1,E_2,\dots$ are events in the same probability space, we write $E_n\to E$ a.s.\ to mean that $P(E_n\bigtriangleup E)\to0$ as $n\to\infty$.

\section{Zero-One Law for Markov Chains}\label{sec: zero one law}

The classical Zero-One Law plays an important role in probability theory. This may be illustrated by the fact that it is the very first theorem presented in the well-known advanced textbook on probability theory by D.~Stroock \cite{Stro11}. Before formulating this result, we introduce some notation. Let
$$
X_0,X_1,X_2,\dots
$$
be a sequence of random variables defined on some probability space $(\Omega,\mathcal F,P)$. Let $\mathcal F_{nm}=\sigma(X_n,\dots,X_{m})$ for $0\le n\le m<\infty$, let $\mathcal F_{n\infty}=\sigma(X_n,\dots)$, and let $\mathcal T=\bigcap_{n\ge0}\mathcal F_{n\infty}$ be the tail $\sigma$-algebra.

\begin{fact}\textbf{(Kolmogorov's Zero-One Law)}
Assume that $(X_n)$ is a sequence of independent random variables. If $A\in \mathcal T$, then $P(A)=0$ or $P(A)=1$.
\end{fact}

While the proof of this law is quite simple, it is important to note that the tail $\sigma$-algebra is a very rich object and may have many complicated events. This is true even for fairly simple situations such as repeatedly tossing a coin.

Perhaps, the most natural relaxation of the assumption of independence is to assume that the sequence of random variables forms a Markov chain. In this case, the tail $\sigma$-algebra does not, in general, satisfy a zero-one law and may contain a number of masses, see \cite{Cohn70}, \cite{Iosifescu:1979}, and the references therein. In Blackwell and Freedman (1964) \cite{Blackwell:Freedman:1964} the following result for when a MC satisfies the zero-one law is given.

\begin{fact}\textbf{(Blackwell and Freedman's Zero-One Law)}
Assume that $(X_n)$ is a homogeneous and recurrent MC on a finite or countably infinite state space such that $P(X_0=i)=1$ for some $i$. If $A\in \mathcal T$, then $P(A)=0$ or $P(A)=1$.
\end{fact}

If we remove any of the assumptions on the MC, then there will be examples where an event $A\in\mathcal T$ with $0<P(A)<1$ exists. On the other hand, Sonin (1991) \cite{Son91a} showed that something akin to the zero-one law nevertheless holds. This results is true even for general nonhomogeneous MCs with no assumptions made on the initial distribution or on the sequence of transition matrices. We begin by describing the general setup.

Let $(S_n)_{n\in\mathbb Z_+}$ be a sequence of finite or countably infinite state spaces and let $(Z_n)_{n\in\mathbb Z_+}$ be a MC with $Z_n$ taking values in $S_n$. Here by MC we mean that the random sequence $(Z_n)_{n\in\mathbb Z_+}$ satisfies the Markov property. Fix $A\in\mathcal T$ and for $0\le a\le b\le1$ let
$$
S_n(a,b) = \{i\in S_n : a\le P(A|Z_n=i)\le b\}.
$$

\begin{thrm}\label{thrm: main}
If $0<p<q<1$, then the following hold:

a) $\lim_{n\to\infty}P(Z_n\in S_n(q,1))=P(A)$,

b) $\lim_{n\to\infty}P(Z_n\in S_n(p,q))=0$, and

c) $\lim_{n\to\infty}P(Z_n\in S_n(0,p))=1-P(A)$.
\end{thrm}

\begin{remark}\label{remark: main}
From the proof we will see that the convergence is not just of probabilities, but of events. Specifically, for any $0<p<q<1$ we have the stronger result that

$ a')$ $(Z_n\in S_n(q,1))\to A$ a.s.,

$b')$ $(Z_n\in S_n(p,q))\to\emptyset$ a.s., and

$c')$ $(Z_n\in S_n(0,p))\to A^c$ a.s.
\end{remark}

Theorem \ref{thrm: main} means that, for large $n$, with probability near one, the trajectories of $(Z_n)$ are in states $i$, where $P(A|Z_n=i)$ is either near $0$ or near $1$. In \cite{Son91a} it is stated that this results ``may be known but we know of no reference.'' At this point we still have not seen a result of this type formulated elsewhere.
However, we note that related ideas appear in \cite{Cohn:1974} and \cite{Iosifescu:1979}. The proof of Theorem \ref{thrm: main}, as given in \cite{Son91a}, is incomplete and has gaps. 
In Sections \ref{sec: tech proof} and \ref{sec: const proof} we give two detailed proofs. The first is simpler, but uses heavy machinery that obscures the underlying mechanics. The second is longer but constructive. It helps to illuminate how the $\sigma$-algebras $\mathcal F_{kn}$  and $\mathcal F_{k\infty}$ converge to the tail $\sigma$-algebra $\mathcal T$.

\section{Zero-One Law for the Entrance $\sigma$-Algebra and a Dichotomy in the Definition of a MC}\label{sec: entrance}

In this section we extend the zero-one law for MCs to the entrance $\sigma$-algebra. We begin the discussion in a more general context. Let
$$
\dots, X_{-2}, X_{-1},X_0,X_1,X_2,\dots
$$
be a sequence of random variables indexed by $\mathbb Z$. The so-called {\em entrance $\sigma$-algebra} is defined by $\mathcal H=\bigcap_{n\le0}\mathcal F_{-\infty n}$, where $\mathcal F_{-\infty n}=\sigma(\dots,X_n)$. For $n\in\mathbb Z_+$, define $Y_n=X_{-n}$ and note that $\mathcal H$ is the tail $\sigma$-algebra for $(Y_n)$. Thus the entrance $\sigma$-algebra is really just a tail $\sigma$-algebra, but when we change the arrow of time and run the sequence backwards. It follows that the entrance $\sigma$-algebra of $(X_n)$ satisfies a zero-one law if and only if the tail $\sigma$-algebra of $(Y_n)$ satisfies it. In the simplest case when $(X_n)$ is a sequence of independent random variables, then so is $(Y_n)$ and it follows that Kolmogorov's zero-one law (Fact 1) holds for the entrance $\sigma$-algebra.

We now turn the case of interest. Let $(Z_n)_{n\in\mathbb Z}$ be a MC and, as before, let $Y_n=X_{-n}$ for $n\in\mathbb Z_+$. We refer to $(Z_n)$ as the MC in forward time and to $(Y_n)$ as the MC in reverse time. It is well-known that the sequence $(Y_n)$ is also a MC. In fact, a simple application of Bayes' rule gives the following result.

\begin{prop}\label{prop: inv time} If $(Z_n)_{n_0<n<n_1}$, where $-\infty\le n_0<n_1\le\infty$, is a  MC, then 
for any integers $n_0<n\le s<n_1$
\begin{eqnarray}\label{eq: inv time matrix}
P(Z_n=i_n| Z_{n+1}=i_{n+1}, \dots,Z_s=i_s) &=& P(Z_n=i_n| Z_{n+1}=i_{n+1})\nonumber \\
&=&P(Z_{n+1}=i_{n+1}| Z_{n}=i_{n})\frac{P(Z_{n}=i_{n})}{P(Z_{n+1}=i_{n+1})}.
\end{eqnarray}
\end{prop}

Since $(Y_n)_{n\in\mathbb Z_+}$ is a MC and $\mathcal H$ is its tail $\sigma$-algebra, the zero-one law for MCs, i.e.\ Theorem \ref{thrm: main} and Remark \ref{remark: main}, remains true for $A\in\mathcal H$ so long as we take $n\to-\infty$. In a similar way, we can define a MC indexed by $\mathbb Z_-$ and extend the zero-one law for MCs to that case.

It may be interesting to note that, when the assumptions of Blackwell and Freedman's zero-one law (Fact 2) hold for $(Z_n)$, it does not guarantee that they will hold for $(Y_n)$. This is because one of the assumptions is that the MC is homogeneous, but, as is clear from \eqref{eq: inv time matrix}, when $(Z_n)$ is homogeneous, $(Y_n)$, in general, is not. While this lack of homogeneity, or equivalently of stationary transitions, in reverse time is well-known, it may nevertheless appear to be surprising. As Hunt (1960) \cite{Hunt:1960} points out, ``In view of the symmetry of past and future in the notion of Markoff chain, the lack of such symmetry in defining Markoff chains with stationary transitions must puzzle many a probabilist.''  In fact, this asymmetry is even stronger and gets to the very heart of how MCs are defined. 

There are two standard definitions of a MC. The first is the one that is used in this paper and in many other places including the classic textbook \cite{Kemeny:Snell:Knapp:1976}. This definition assumes that a MC is a sequence of random variables that satisfies the Markov property. The other definition, which is given in many if not most textbook, see e.g.\ \cite{Iosifescu:1980}, is to start with a sequence of transition matrices $(P_n)$, where $P_n$ governs the transitions at time $n$.  In the homogeneous case all of these matrices are equal to one transition matrix $P$.  We refer to the collection of matrices $(P_n)$, or to $P$ in the homogeneous case, as a {\em Markov Chain model} (MCM). This model does not define one MC, but a family of MCs, each determined by an initial distribution. When we fix an initial distribution, we fix the specific Markov chain.

In forward time these two definitions are almost equivalent and, for this reason, not much attention is generally paid to the difference. However, their equivalence breaks down for MCs in reverse time. To see this, consider a MC in forward time that is governed by some MCM. From \eqref{eq: inv time matrix} it is clear that the transition matrices of the MC in reverse time depend on the initial distribution and so, in general, no MCM can exist in reverse time. There is an important exception, which is often used to circumvent this issue, see e.g.\ \cite{Dynkin:1969}. If a MC is both homogeneous and stationary, then a MCM will exist in both forward and backward time, although the transition matrices may be different. They are the same only under additional assumptions, which lead to the so-called reversible MCs.

\section{Existence of Markov Chains on $\mathbb Z_-$ and $\mathbb Z$} \label{sec: MC past}

In the previous section we discussed MCs indexed by $\mathbb Z_-$ and $\mathbb Z$. However, we did not consider the question of whether such MCs exist. This is not a trivial question because such a MC does not have a starting point in time; it starts at ``minus infinity.'' Thus there is no initial distribution. Note that we are not talking about the MC going in a backwards direction, the arrow of time goes, as usual, from left to right.  The question of when such a MC exists was first posed by Kolmogorov in \cite{Kolm1936} and  was formulated as follows: Given a sequence of stochastic matrices $(P_n)_{n\in\mathbb N_-}$, does there exist a MC $(Z_n)_{n\in\mathbb Z_-}$ with these as its transition matrices? We distinguish three cases:
\begin{itemize}
\item Finite constant: the number of states in each state space is finite and equal to some integer $N$. 
\item Finite: the number of states in each state space is finite, but may approach infinity as $n\to-\infty$.
\item Countable: all of the state spaces have a countably infinite number of states.
\end{itemize}
To the best of our knowledge only the finite constant case has been considered in the literature, see \cite{Kolm1936} and \cite{Bla1945}. In this case, Kolmogorov \cite{Kolm1936} showed that a MC always exists and gave a necessary and sufficient condition for uniqueness. However, the proofs in \cite{Kolm1936} are not very detailed. We give detailed proofs, which hold not only in the finite constant case but in the more general finite case. Our proofs do not seem to be exactly what Kolmogorov had in mind, but they are along similar lines. We also consider the countable case, where we show that there are situations when a MC does not exist and give a general sufficient condition for when it does. Throughout, our focus is on the case of MCs indexed by $\mathbb Z_-$. However, all of our results immediately extend to MCs indexed by $\mathbb Z$ since MCs indexed by $\mathbb Z_+$ always exist.

\subsection{Finite Case}

Let $(S_n)_{n\in\mathbb Z_-}$ be a sequence of finite state spaces with $|S_n|=N_n<\infty$ with $\liminf_{n\to-\infty}N_n\le \infty$. For simplicity of notation and without loss of generality, we identify $S_n$ with the set $\{1,2,\dots,N_n\}$. For $n\leq 0$ let
$$
D(n) = \left\{x\in\mathbb R^{N_n}: \sum_{i=1}^{N_n} x(i)=1 \mbox{ and }x(1),\dots,x(N_n)\ge0\right\}
$$
be the probability simplex, i.e.\ the collection of all probability measures on $S_n$. Let $(P_n)_{n\in\mathbb N_-}$ be a sequence of stochastic matrices, with $P_n$ being an $N_{n}\times N_{n+1}$ matrix representing the transition from time $n$ to time $n+1$. For $s<t$, define multistep transition matrices by $P_{st}=\prod_{n=s}^{t-1} P_n$. Note that $P_n=P_{n,n+1}$ and that
\begin{eqnarray}\label{eq: Pst}
P_{st}=P_{su}P_{ut}, \ \ s\le u\le t.
\end{eqnarray}
We identify these matrices with the linear transformations $P_{st}:\mathbb R^{N_s}\mapsto \mathbb R^{N_t}$ given by $P_{st}(m)=mP_{st}$ for $m\in\mathbb R^{N_s}$. Note that we use $P_{st}$ to represent both the matrix and the corresponding linear transformation, but this should not cause any confusion. 

The problem of interest is to determine whether there exists a MC $(Z_n)_{n\in \mathbb Z_-}$ with $Z_n$ taking values in $S_n$, which is governed by the sequence of transition matrices $(P_n)_{n\in\mathbb N_-}$. Since we are starting at ``minus infinity,'' there is no initial distribution in this case. By Kolmogorov's extension theorem, the problem is equivalent to asking if there exists a sequence of vectors $(m_n)_{n\in \mathbb Z_-}$ such that 
\begin{eqnarray}\label{eq: defn m}
m_n\in D(n) \mbox{ and } m_{n+1}= P_{n}(m_n)=m_nP_n,\ \  n\in\mathbb N_-.
\end{eqnarray}
In this case, $P(Z_n=i)=m_n(i)$ and for $k<n$ we have $P(Z_k=i_{k},Z_{k+1}=i_{k+1},\dots,Z_n=i_n)=m_k(i_k)p_{k}(i_k,i_{k+1})\cdots p_{n-1}(i_{n-1},i_{n})$, where for $\ell\in\mathbb N_-$, $p_\ell(i,j)$ denotes the element of matrix $P_{\ell}$ that is located in the $i$th row and $j$th column.

\begin{thrm}\label{thrm: exist finite}
At least one sequence of vectors $(m_n)_{n\in\mathbb Z_-}$ satisfying \eqref{eq: defn m} exists.
\end{thrm}

With essentially the same proof, we get the following more general result, which will be useful for proving an analogous theorem in the countable case.

\begin{thrm}\label{thrm: gen}
Let $(V_n)_{n\in\mathbb Z_-}$ be a sequence of metric spaces and let $(P_n)_{n\in\mathbb N_-}$ be a sequence of continuous transformations with $P_n:V_{n}\mapsto V_{n+1}$. 
Assume that there is a sequence of nonempty {\em compact} sets $(D(n))_{n\in\mathbb Z_-}$ with $D(n)\subset V_n$ such that the image of $D(n)$ under $P_n$ is contained in $D(n+1)$, i.e.\ $P_n(D(n))\subset D(n+1)$. In this case, there exists a sequence of points $(m_n)_{n\in\mathbb Z_-}$ with $m_n\in D(n)$ such that 
$$
m_{n+1} = P_n(m_{n}),\ \  n\in\mathbb N_-.
$$
\end{thrm}

Before giving our next results, we set up some notation. For $s<t\le 0$, define the set $\Delta(s,t)\subset D(t)$ to be the image of $D(s)$ under the linear transformation $P_{st}$. Note that, by \eqref{eq: Pst}, for any $s<t\le0$
\begin{eqnarray}\label{eq: subsetseq}
\dots\subset \Delta(s,t)\subset\dots\subset \Delta(t-1,t)\subset D(t).
\end{eqnarray}
The set $D(s)$ is a simplex with $N_s$ vertices. More specifically, it is the convex hull of $e_1^{(N_s)}, e_2^{(N_s)}, \dots, e_{N_s}^{(N_s)}$, the rows of the $N_s\times N_s$-dimensional identity matrix. Similarly, $\Delta(s,t)$ is the convex hull of $a_1,a_2,\dots,a_s$, where $a_k=P_{st}(e_k^{(N_s)})$, $k=1,2,\dots,N_s$. The meanings of $\Delta(s,t)$ and $a_k$ are as follows. If we start running MC $(Z_n)_{s\le n\le t}$ at time $s$ with transition matrices $(P_{n})_{s\le n\le t-1}$, then any distribution in $D(s)$ can serve as the distribution of $Z_s$, i.e.\ as the initial distribution. However, at time $t$, the only possible distributions for $Z_t$ are those in $\Delta(s,t)$. In particular $a_k\in\Delta(s,t)$ corresponds to the distribution that we get at time $t$ if the initial distribution at time $s$ was $e_k^{(N_s)}$. If we let the time $s$ at which we start the MC approach $-\infty$, then the possible distributions of $Z_t$ will be contained in $\Delta(t):=\bigcap_{s< t}  \Delta(s,t)$. In fact, in the proof of Theorem \ref{thrm: exist finite} we show that $\Delta(t)$ is exactly the set of all possible distributions of $Z_t$. It follows that such a MC on $\mathbb Z_-$ exists if and only of $\Delta(0)\ne\emptyset$. We now characterize when the distribution at time $t$ is unique under an additional assumption. This extends the second theorem in \cite{Kolm1936} and Corollary 2 in \cite{Bla1945}, which were both formulated for the finite constant case.

\begin{thrm}\label{thrm: delta is simplex}
Let $N=\liminf_{t\to-\infty}N_t $ and assume that $N<\infty$.\\
a) For any $t\le0$, $\Delta(t)$ is a simplex with at most $N$ vertices.  More specifically, there exists an $N\times N_t$-dimensional stochastic matrix $P_t^*$ and a sequence $(s_n)\subset\mathbb Z_-$ with $s_n\to-\infty$ and $P_{s_nt}\to P^*_t$ such that $\Delta(t)$ is the convex hull of $\{a_1,a_2,\dots,a_{N}\}$, where $a_i = P_t^*(e_i^{(N)})$.\\
b) There is a unique $m_t\in \Delta(t)$ at time $t$ if and only if for any matrix $P_t^*$, which satisfies $\lim_{n\to\infty}P_{s_nt}=P_t^*$ for some subsequence $(s_n)$, all row of $P_t^*$ are equal to $m_t$.\\
c) Assume, in addition, that $\lim_{t\to-\infty}N_t=N<\infty$ exists. In this case, there is a unique distribution at time $t$ if and only if there exists a matrix $P_t^*$ with $\lim_{s\to-\infty}P_{st}=P_t^*$ and all row of $P_t^*$ are identical and equal to this unique distribution.
\end{thrm}

Note that, in a), if we consider different subsequences, we may get different matrices $P_t^*$. However, their ranges will, necessarily, be the same. This is true even if all of the state spaces have the same number of elements.  A simple example is when $P_n=\begin{pmatrix}
0 &\ \ &1 \\
1& \ \ &0
\end{pmatrix}$ for each $n<0$. In b), if $\lim_{t\to-\infty}N_t$ does not exist, then $\lim_{s\to-\infty}P_{st}=P_t^*$ will not exist. An example is when $P_{2n}=\begin{pmatrix}
.5 &\ \ &.5
\end{pmatrix}$ and  $P_{2n+1}=\begin{pmatrix}
1 \\
1
\end{pmatrix}$ for $n<0$. In this case if $t$ is even, then the possible limits are $P_t^*=\begin{pmatrix}
.5 &\ \ & .5 \\
.5 &\ \ &.5
\end{pmatrix}$ and $P_t^*=\begin{pmatrix} .5 &\ \ & .5 \end{pmatrix}$, and if $t$ if odd then the possible limits are $P_t^*=\begin{pmatrix}
1 \\
1
\end{pmatrix}$ and $P_t^*=\begin{pmatrix}1\end{pmatrix}$.

\subsection{Countable Case}

We now turn to the countable case, where $|S_n|=\infty$ for each $n\le0$. For simplicity of notation and without loss of generality, we assume that each $S_n=\mathbb N=\{1,2,\dots\}$. As usual, let $\ell^1$ be the space of absolutely summable sequences $m=(m(1),m(2),\dots)$ of real numbers equipped with the norm
$$
\|m\|=\sum_{i=1}^\infty |m(i)|.
$$
Let
$$
D=\left\{m\in\ell_1: \|m\|=1 \mbox{ and } m(i)\ge0 \mbox{ for each }i\ge0\right\}
$$
be the set of probability measures in $\ell^1$. Let $P$ be an infinite stochastic matrix, see e.g.\ \cite{Kemeny:Snell:Knapp:1976} for details about such matrices. For $1\le i,j<\infty$, let $p(i,j)$ be the element of $P$ in the $i$th row and $j$th column, let $p(i,\star)$ be the $i$th row of $P$, and let $p(\star,j)$ be the $j$th column of $P$. The fact that $P$ is stochastic means that all of its rows belong to $D$. $P$ corresponds to the linear transformation $P:\ell^1\mapsto\ell^1$ such that for any $m\in\ell^1$, $m'=P(m)$ is the vector with
$$
m'(j) = \sum_{i=1}^\infty m(i) p(i,j).
$$
In this case
\begin{eqnarray}\label{eq: bounded trans}
\|m'\|=\sum_{j=1}^\infty |m'(j)| \le \sum_{i=1}^\infty |m(i)| \sum_{j=1}^\infty p(i,j) = \sum_{i=1}^\infty |m(i)|  = \|m\|,
\end{eqnarray}
which implies that $P$ is bounded, and thus that it is continuous, see e.g.\ Proposition 5.2 in \cite{Folland:1999}. Note further that, if $m(i)\ge0$ for each $i$, then by arguments similar to those in \eqref{eq: bounded trans}, we have $\|m'\| = \|m\|$.

Let $(P_n)_{n\in\mathbb N_-}$ be a sequence of infinite stochastic matrices. For $s< t\le0$ let $P_{st}=P_{t-1}\circ P_{t-2}\circ\cdots\circ P_{s}$, where $\circ$ denotes the composition operator. Each $P_{st}$ corresponds to an infinite stochastic matrix and it is easily checked that 
\begin{eqnarray}\label{eq: comp stoch matrix}
p_{st}(i,k) = \sum_{j=1}^\infty p_{su}(i,j)p_{ut}(j,k), \ \ \ s\le u\le t.
\end{eqnarray}
For $s<t\le0$, let $\Delta_n(s,t)\subset D$ be the image of $D$ under $P_{st}$.

As in the finite case, to show that a MC $(Z_n)_{n\in\mathbb Z_-}$ with transition matrices $(P_n)_{n\in\mathbb N_-}$ exists, it suffices to show that there exists a sequence $(m_n)_{n\in\mathbb Z_-}\subset D$ with
\begin{eqnarray}\label{eq: defn m count}
m_{n+1} = P_n(m_{n}) , \ \ \ n\in\mathbb N_-. 
\end{eqnarray}
We can use Theorem \ref{thrm: gen} to find when such a sequence exists. However, we must be careful. Unlike in the finite case, here the set $D$ is not compact. This follows from Prohorov's Theorem (see e.g.\ Theorem 25.10 in \cite{Billingsley:1995}), which says that subsets of $D$ are compact if and only if they are closed and tight. A statement of this result in the context of the larger space $\ell^1$ can be found in Theorem 44.2 of \cite{Treves:1967}. For this reason, in order to define the appropriate compact sets, we require an additional assumption. First, we give a definition.

\begin{defn}
Fix $H\subset D$. If for any $\varepsilon>0$ there exists an $N_\varepsilon>0$ such that for any $m\in H$
$$
\sum_{k=1}^{N_\varepsilon} m(k)\ge 1-\varepsilon,
$$
we say that $H$ is tight. If $H$ corresponds to the rows of an infinite stochastic matrix $P$, then we say that $P$ is tight. 
\end{defn}

With this we can state our assumption. We call this Condition P because it is a version of the condition in Prohorov's Theorem.\\

\noindent\textbf{Condition P.} {\it There exists an infinite set $V\subset \mathbb N_-$ such that for each $n\in V$, the infinite stochastic matrix $P_{n-1}$ is tight, i.e.\ for any $n\in V$ and $\varepsilon>0$ there exists an $N_\varepsilon(n)>0$ such that for any $i\in\mathbb N$
$$
\sum_{k=1}^{N_\varepsilon(n)} p_{n-1}(i,k)\ge 1-\varepsilon.
$$}

Condition P means that infinitely many of the $P_{n}$s are tight. Note that, in the condition, we allow $\liminf_{n\to-\infty,n\in V}N_\varepsilon(n)\le\infty$. While the condition is stated in terms of $P_{n-1}$, it implies that, for any $n\in V$ and any $k<n$, $\Delta(k,n)$ is a compact set, see Lemma \ref{lemma: delta compact} in Section \ref{sec: proofs} below. We now give our main result for the countable case.

\begin{thrm}\label{thrm: exist countable}
If Condition P holds, then at least one sequence $(m_n)_{n\in\mathbb Z_-}\subset D$ satisfying \eqref{eq: defn m count} exists.
\end{thrm}

Note that this theorem shows only that Condition P is sufficient. In fact, it is not necessary. In Section \ref{sec: proofs} we give two examples where Condition P does not hold. The first is the example of a symmetric random walk on $\mathbb Z$, for which we show that a solution to \eqref{eq: defn m count} does not exist. The second is the situation where each $P_n$ is onto, in which case we show that a solution to \eqref{eq: defn m count} always exists. 

Next, we turn to the question of when the solution to \eqref{eq: defn m count} is unique. To state this result, we need a stronger condition, which is a uniform version of Condition P.\\

\noindent\textbf{Condition U.} {\it There exists an infinite set $V\subset \mathbb N_-$ such that for any $\varepsilon>0$ there exists an $N_\varepsilon$ such that for any $n\in V$ and $i\in\mathbb N$
$$
\sum_{k=1}^{N_\varepsilon} p_{n-1}(i,k)\ge 1-\varepsilon.
$$}

For $t\le0$, let $\Delta(t)=\bigcap_{s<t} \Delta(s,t)$. By arguments similar to those in the proof of Theorem \ref{thrm: exist finite}, $\Delta(t)$ is exactly the set of all vectors that can serve as solutions at time $t$, i.e.\ it is the set of all $m\in D$ with the property that there exists a sequence $(m_n)_{n\in\mathbb Z_-}\subset D$ such that $m_t=m$ and \eqref{eq: defn m count} holds. We now give our uniqueness results.

\begin{thrm}\label{thrm: uniq count}
Assume that Condition U holds. \\
a) The set $\Delta(t)$ is the convex hull of an at most countable collection of vectors in $D$. \\
b) There is a unique distribution at time $t$ if and only if there exists and infinite stochastic matrix $P_t^*$ with $\lim_{s\to-\infty}P_{st}=P_t^*$ such that all row of $P_t^*$ are identical and equal to this unique distribution.
\end{thrm}

This result extends Theorem \ref{thrm: delta is simplex} and the second theorem in \cite{Kolm1936} to the countable case.

\section{Proofs}\label{sec: proofs all}

\subsection{Proof of Theorem \ref{thrm: main} using L\'evy's Theorem} \label{sec: tech proof}

In this section we give a proof of Theorem \ref{thrm: main} using L\'evy's `Upward' Theorem. For a proof of Kolmogorov's zero-one law using a similar approach see, e.g.,\ Section 14.3 in \cite{Williams:1991}. L\'evy's `Upward' Theorem (see 14.2 in \cite{Williams:1991}) states that for any random variable $X$ with a finite expectation
\begin{eqnarray}\label{eq: Levy upward}
\lim_{n\to\infty} \rE[X|\mathcal F_{0n}] = \rE[X|\mathcal F_{\infty}] \ \mbox{a.s.}
\end{eqnarray}
where 
$
 \mathcal F_\infty = \sigma\left(\bigcup_{n\ge0} \mathcal F_{0n}\right)
$.

\begin{lemma}
We have $\mathcal T\subset \mathcal F_\infty$.
\end{lemma}

\begin{proof}
Since $Z_k$ is measurable $\mathcal F_\infty$ for each $k$, it follows that $\mathcal F_{n\infty}\subset\mathcal F_\infty$ for each $n$. Thus, $\mathcal T=\bigcap_{n\ge0}\mathcal F_{n\infty}\subset\mathcal F_\infty$.
\end{proof}

For any $A\in\mathcal T$ we have
$$
\lim_{n\to\infty} P(A|Z_n) = \lim_{n\to\infty} P(A|\mathcal F_{0n}) = P(A|\mathcal F_{\infty}) = I_A  \ \mbox{a.s.}
$$
where the first equality follows by the Markov property, which is applicable since $A\in\mathcal T\subset\mathcal F_{(n+1)\infty}$, the second by \eqref{eq: Levy upward} applied to the random variable $X=I_A$, and the third by the fact that $A\in\mathcal T\subset\mathcal F_{\infty}$. From here, part $a$) of Theorem \ref{thrm: main} and part $a'$) of Remark \ref{remark: main} follow immediately from the following lemma. After that, the remaining parts easily follow. 

\begin{lemma}
If $0<q<1$, then
$$
\{Z_n\in S_n(q,1)\}\to A  \ \mbox{a.s.}
$$
\end{lemma}

\begin{proof}
For simplicity of notation, we suppress the dependence on $q$ and $A$ and denote $B_n =\{Z_n\in S_n(q,1)\}$. Let $Y_n$ be a version of $P(A|Z_n)$ such that
$$
\lim_{n\to\infty} Y_n(\omega) = I_A(\omega)\mbox{ for each } \omega\in\Omega.
$$
Note that $A=\{\omega:\lim_{n\to\infty} Y_n(\omega)=1\}$. Thus, for any $\omega\in A$ there exists an $N(\omega)$ such that, if $n\ge N(\omega)$, then
$$
\left|Y_{n}(\omega)-1\right|<1-q.
$$
It follows that $\omega\in B_{n}$ for all $n\ge N(\omega)$ and so $A\subset\liminf_{n\to\infty} B_n$. Next, note that $A^c=\{\omega:\lim_{n\to\infty} Y_n(\omega)=0\}$. We can similarly show that $A^c\subset\liminf_{n\to\infty} \left(B_n\right)^c$. It follows that
$$
A \subset \liminf_{n\to\infty} B_n \subset \limsup_{n\to\infty} B_n \subset A,
$$
which guarantees that the limit of $B_n$ exists and equals $A$. Since we chose $Y$ to be a particular version of $P(A|Z_n)$, in general, the result only holds almost surely.
\end{proof}

\subsection{Constructive Proof of Theorem \ref{thrm: main}} \label{sec: const proof}

In this section in give a proof based on approximating events in the tail by events in $\mathcal F_{kn}$. A similar approach can be used to prove Kolmogorov's zero-one law and is standard in proofs of the Hewitt-Savage zero-one law, see e.g.\ the proof of Theorem 36.5 in \cite{Billingsley:1995}. Our proof uses the following well-known facts. The first is a version of Corollary 1 on page 169 in \cite{Billingsley:1995} and the second is easy to show.

\begin{prop}\label{prop: approx}
 For any $k$, any $\varepsilon >0$, and large enough $n$, there exists a set $A_{kn}\in \mathcal F_{kn}$, such that $P(A\bigtriangleup A_{kn})<\varepsilon$. 
 \end{prop}

\begin{prop}\label{prop: z}
Fix $\alpha,\varepsilon>0$. If $P(A\bigtriangleup B)<\alpha$ and $P(B\bigtriangleup C)<\varepsilon$, then $P(A\bigtriangleup C)<\alpha+\varepsilon$.
\end{prop}

By Proposition \ref{prop: approx}, there are $A_{kn}\in \mathcal F_{kn}$ with $\lim_k\lim_n A_{kn}=A$ a.s. For any $0\le a\le b\le1$, let $S_{kn}(a,b) = \{i\in S_n:a\le P(A_{kn}|Z_n=i)\le b\}$. For simplicity of notation, we suppress the dependence on $q$, $A$, and $A_{kn}$ and define
\begin{eqnarray*}
S_{n}=S_{n}(q,1), \ S_{kn}=S_{kn}(q,1), \
B_{kn}=(Z_n\in S_{kn}), \mbox{ and } B_{n}=(Z_n\in S_{n}).
\end{eqnarray*}
We now give an approximation result, which is fundamental to the proof and may be of independent interest.

\begin{lemma}\label{lemma: approx}
Fix $A\in \mathcal T$ and $A_{kn}\in \mathcal F_{kn}$ with $\lim_k\lim_n A_{kn}=A$ a.s. If $0<p<q<1$, then:

$a)$ $\lim_k\lim_nP(Z_n\in S_{kn}(q,1))=P(A)$,

$b)$ $\lim_k\lim_nP(Z_n\in S_{kn}(p,q))=0$,

$c)$ $\lim_k\lim_nP(Z_n\in S_{kn}(0,p))=1-P(A)$.
\end{lemma}

\begin{remark}\label{remark: approx}
As with Theorem \ref{thrm: main}, from the proof we will see that the following stronger result holds. For any $0<p<q<1$ we have $a')$
$\lim_k\lim_n(Z_n\in S_{kn}(q,1))= A$ a.s., $b')$ $\lim_k\lim_n(Z_n\in S_{kn}(p,q))=\emptyset$ a.s., and $c')$ $\lim_k\lim_n(Z_n\in S_{kn}(0,p))= A^c$ a.s.
\end{remark}

\begin{remark}
Note that, at time $n$, event $A_{kn}$ is from the past, while event $A\in\mathcal T$ is from the future. Thus, Lemma \ref{lemma: approx} gives a {\em backwards} formulation of the problem, while Theorem \ref{thrm: main} gives a {\em forwards} formulation.
\end{remark}

\begin{proof}[Proof of Lemma \ref{lemma: approx}]
First, note that, since a) is for any $q\in(0,1)$, it immediately gives b), and then, a) and b) together give c). Thus, it suffices to prove a). In fact, as we will show, it suffices to prove
\begin{eqnarray}\label{eq: to show}
\lim_k\lim_nP(AB_{kn})=P(A).
\end{eqnarray}

We begin by showing that \eqref{eq: to show} implies a). For the moment, assume that \eqref{eq: to show} holds for every $A\in\mathcal T$ and every $q\in(0,1)$. This immediately gives $\lim_k\lim_nP(AB^c_{kn})=0$. Further, noting that $P(A|Z_n=i)< q$ is equivalent to $P(A^c|Z_n=i)\geq 1-q$ and applying \eqref{eq: to show} with $A^c$ in place of $A$ and $1-q$ in place of $q$, we get $\lim_k\lim_nP(A^cB^c_{kn})=P(A^c)$, and hence $\lim_k\lim_nP(A^cB_{kn})=0$. It follows that
\begin{eqnarray}\label{eq: triangle to zero}
P(A \bigtriangleup B_{kn})=P(A^cB_{kn})+P(AB^c_{kn})\to0,
\end{eqnarray}
which easily gives a). It remains to verify \eqref{eq: to show}.

Fix $0\le k< n< s<\infty$ and for simplicity set $D_i:=(Z_n=i)$. We have
\begin{eqnarray*}
P(A_{kn}B^c_{kn}A_{ns}) &=& \sum_{i\in S^c_{kn}} P( A_{kn}D_{i}A_{ns}) = \sum_{i\in S^c_{kn}} P( A_{kn}|D_{i}A_{ns})P(A_{ns}D_i)\\
&=& \sum_{i\in S^c_{kn}} P( A_{kn}|D_{i})P(A_{ns}D_i)\\
&\le& q \sum_{i\in S^c_{kn}} P(A_{ns}D_i)=qP(B^c_{kn}A_{ns}),
\end{eqnarray*}
where the second line follows by Proposition \ref{prop: inv time} and the third line by the definition of $S^c_{kn}$. Letting $s\to\infty$, we obtain $P(A_{kn}B^c_{kn}A_{n\infty})\leq qP(B^c_{kn}A_{n\infty})$. Now, noting that $\lim_k\lim_nA_{kn}=A$ a.s.\ and $\lim_nA_{n\infty}=A$ a.s.\ gives $\lim_k\lim_nP(B^c_{kn}A)\leq q\lim_k\lim_nP(B^c_{kn}A)$. Since $q<1$, we have $\lim_{k}\lim_{n}P(B^c_{kn}A)=0$. Finally, since $P(A)=\lim_k\lim_n[P(AB_{kn})+P(AB^c_{kn})]$, we get $\lim_k\lim_nP(AB_{kn})=P(A)$, which is \eqref{eq: to show}.
\end{proof}

We will use Lemma \ref{lemma: approx} to prove Theorem \ref{thrm: main} by approximating $A$ by $A_{kn}$. Most of the heavy lifting is done by the following lemma.

\begin{lemma}\label{lemma: setmin}
If $A\in\mathcal T$, $q\in(0,1)$, and $\varepsilon>0$, then for large enough $k$ and $n$
$$
P(B_{kn}\setminus B_n)<\varepsilon
\mbox{ and }
P(A\setminus B_n)<\varepsilon.
$$
\end{lemma}

\begin{proof}
We begin by writing
\begin{eqnarray*}
P(B_{kn}) &=& \sum_{i\in S_{kn}S_n}P(Z_{n}=i) + \sum_{i\in (S_{kn}\setminus S_n)}P(Z_{n}=i)=:a+b,\\ \notag
P(AB_{kn})&=& \sum_{i\in S_{kn}S_n}P(Z_{n}=i)P(A|Z_n=i) + \sum_{i\in (S_{kn}\setminus S_n)}P(Z_{n}=i)P(A|Z_n=i)=:c+d.
\label{eq: ab}
\end{eqnarray*}
Clearly, $c\le a$ and, by the definition of $S_n$, $d\le q\sum_{i\in (S_{kn}\setminus S_n)}P(Z_{n}=i)=qb$. Therefore $P(AB_{kn})\leq a +qb$. For any $\alpha>0$, \eqref{eq: triangle to zero} implies that for large $k$ and $n$ we have $P(AB_{kn})>P(A)-\alpha$ and $P(B_{kn})<P(A)+\alpha$. It follows that $a+b<P(A)+\alpha$, $P(A)-\varepsilon<a+qb$, thus $P(B_{kn}\setminus B_{n})=b<2\alpha /(1-q)$. Since $q\in(0,1)$, this is less than $\varepsilon$ for an appropriate choice of $\alpha$, which gives the first part. Now, combining this with the fact that, by \eqref{eq: triangle to zero}, for large enough $k$ and $n$ we have $P(A\setminus B_{kn})<\varepsilon/2$ and the fact that $A\setminus B_{n} \subset (A\setminus B_{kn})\cup(B_{kn}\setminus B_{n})$  completes the proof.
\end{proof}

\begin{proof}[Proof of Theorem \ref{thrm: main}]
As in the proof of Lemma 3, b) and c) follow immediately from a), so we just need to prove a).  In fact we will prove 
that for any $\varepsilon>0$ and large enough $n$
$$
P(A\bigtriangleup B_n)=P(A\setminus B_n)+P(B_n\setminus A)<\varepsilon.
$$
This result is stronger than what is needed for Theorem \ref{thrm: main} and will give us the stronger results formulated in Remark \ref{remark: main}. By Lemma \ref{lemma: setmin}, for large enough $n$ we have $P(A\setminus B_n)<\varepsilon/2$. Next, note that $P(B_n\setminus A)=P(A^c\setminus B^c_n)$ and that $B_n^c=(Z_n\in S^*_{n}(1-q,1))$, where $S^*_n(1-q,1)=\{i\in S_n: 1-q\le P(A^c|Z_n=i)\}$. Thus, we can apply Lemma \ref{lemma: setmin} with $A^c$ and $1-q$ to get that for large enough $n$, $P(B_n\setminus A)<\varepsilon/2$.
\end{proof}

\subsection{Proofs for Theorems in Section \ref{sec: MC past}} \label{sec: proofs}

\begin{proof}[Proof of Theorem \ref{thrm: exist finite}]
For each $s$, $D(s)$ is a compact set and, since $P_{st}:\mathbb R^{N_s}\mapsto \mathbb R^{N_t}$ is a continuous transformation and images of compact sets under continuous transformations are compact (see e.g.\ Theorem 26.5 in \cite{Munkres:2000}), $\Delta(s,t)$ is compact for all $s<t\le0$. Since $\Delta(s,t)\ne\emptyset$ for each $s< t$ and \eqref{eq: subsetseq} holds, Cantor's Intersection Theorem (see Theorem 26.9 in \cite{Munkres:2000}) implies that $\Delta(t)\ne\emptyset$ for each $t\le0$.

We next apply a diagonal process to show that, if $m_t\in\Delta(t)$, then there exists an $m_{t-1}\in\Delta(t-1)$ with $m_t=P_{t-1}(m_{t-1})$. To see this, first note that by the definition of $\Delta(t)$, $m_t\in \Delta(s,t)$ for all $s<t$. Thus, for any  $s<t-1$, there exists an $m_{t-1}(s)\in\Delta(s,t-1)$ such that $P_{t-1}(m_{t-1}(s)) = m_{t}$. Since $\Delta(s,t-1)$ is compact for each $s<t-1$, and \eqref{eq: subsetseq} holds, there exists a sequence $(s_n)\subset\mathbb Z_-$ with $s_n\to-\infty$ such that $m_{t-1}(s_n)\to m_{t-1}$ for some $m_{t-1}\in D(t-1)$ and $m_{t-1} \in \Delta(s,t-1)$ for each $s$. Thus $m_{t-1} \in \Delta(t-1)$. From here, the continuity of $P_{t-1}$ implies that $m_t=\lim_{n\to\infty} P_{t-1}(m_{t-1}(s_n)) = P_{t-1}(m_{t-1})$, as required. From here, a simple inductive argument completes the proof.
\end{proof}

The proof of Theorem \ref{thrm: gen} is similar to that of Theorem \ref{thrm: exist finite} and is thus omitted.

\begin{proof}[Proof of Theorem \ref{thrm: delta is simplex}]
We begin by proving a). Let $M$ be any subsequential limit of $(N_t)_{t\in\mathbb N_-}$ and note that $M\ge N$. Since the $N_t$'s are integers, there is a subsequence $(P_{u_n t})$ such that each $P_{u_nt}$ has $M$ rows. Within this subsequence, the first rows of the matrices form a tight sequence of probability measures (as any sequence of probability measures on a finite set is tight). It follows that there is a subsequence that converges to a probability measure. Applying this idea to the other rows of the matrices and potentially taking further subsequences, shows that there exists a subsequence $P_{s_nt}$ and a stochastic matrix $P_t^*$ with $M$ rows such that $P_{s_{n}t}\to P_t^*$ as $n\to\infty$. Now, fix $m_t\in\Delta(t)$ and note that, from the proof of Theorem \ref{thrm: exist finite}, there exists a sequence $(m_s)_{s<t}$ with $m_s\in\Delta(s)$ and $m_t = P_{st}(m_s)$. Since $\{m_{s_n}\}$ is a tight sequence, there exists a subsequence $m_{s_{n_k}}$ and a probability measure $m\in\mathbb R^M$ with $m_{s_{n_k}}\to m$ as $k\to\infty$. It follows that
$$
m_t=\lim_{k\to\infty} m_{s_{n_k}} P_{s_{n_k}t} = m P_t^*= \sum_{i=1}^M m(i) e_i^{(M)} P_t^*= \sum_{i=1}^M m(i) a_i,
$$
which is a convex combination of the $a_i$s and thus belongs to their convex hull. Conversely, fix a weight vector $p\in\mathbb R^M$ with $p(i)\ge0$ and $\sum_{i=1}^M p(i)=1$, and consider the convex combination $m=\sum_{i=1}^M p(i) a_i$. We have
$$
m=\sum_{i=1}^M p(i) a_n= \sum_{i=1}^M p(i) e_i^{(M)} P_t^* = p P_t^*= \lim_{n\to\infty} p P_{s_{n}t}.
$$
Since, by definition, $p P_{s_{n}t}\in\Delta(s,t)$ for every $s\in[s_n,t-1]$ and $\Delta(s,t)$ is compact for each $s<t$, it follows that the limit $m$ must be in each $\Delta(s,t)$. Thus $m\in\Delta(t)$. The definition of $N$ implies that the above holds with $N$ in place of $M$.

We now turn to b).  Arguments similar to those in a) imply that any $P_t^*$ of the required form is an $M\times N_t$-dimensional stochastic matrix for some $M\ge N$. 
Clearly, the distribution at time $t$ is unique if and only if $|\Delta(t)|=1$. In light of a), this holds if and only if $P_t^*$ maps each $e_i^{(M)}$ to the same vector, which is equivalent to all of the row of $P_t^*$ being the same. It follows that $\Delta(t)$ contains exactly one element, which is this row. 

To prove c), it suffices to show that, under the assumption $|\Delta(t)|=1$, the limit exists. Assume that there are two subsequences with $P_{ s^{(1)}_{n} t }\to P_t^{*1}$ and $P_{ s^{(2)}_{n} t } \to P_t^{*2}$. Tightness arguments similar to those in the proof of a) imply that both $P_t^{*1}$ and $P_t^{*2}$ are stochastic matrices. From here, b) implies that all rows of both $P_t^{*1}$ and $P_t^{*2}$ are equal to the unique vector in $\Delta(t)$. The fact that, in this case, the two matrices have the same dimensions completes the proof. 
\end{proof}

\begin{lemma}\label{lemma: delta compact}
Assume that Condition P holds and let $V$ be the set in that condition.  For all $t\in V$ and all $s<t\le0$, the following hold: 
a) both $P_{st}$ and $\Delta(s,t)$ are tight; b) $\Delta(s,t)$ is compact.
\end{lemma}

\begin{proof}
We begin with a). Fix $s<t\le0$. From \eqref{eq: comp stoch matrix} it follows that 
\begin{eqnarray}\label{eq: which N epsilon}
\sum_{k=1}^{N_\varepsilon(t)}  p_{st}(i,k)= \sum_{j=1}^\infty p_{s,t-1}({i,j}) \sum_{k=1}^{N_\varepsilon(t)} p_{t-1}({j,k})\ge \sum_{j=1}^\infty p_{s,t-1}({i,j})(1-\varepsilon)= 1-\varepsilon,
\end{eqnarray}
which gives the first part. Next, fix $m\in\Delta(s,t)$, let $m'=P_{st}(m)$, and note that
$$
\sum_{k=1}^{N_\varepsilon(t)}  m'(k) = 
\sum_{i=1}^\infty m(i) \sum_{k=1}^{N_\varepsilon(t)} p_{st}({i,k}) \ge \sum_{i=1}^\infty m(i) (1-\varepsilon)=1-\varepsilon,
$$
which gives the second part.

We now turn to b). We will show that $\Delta(s,t)$ is sequentially compact, which is equivalent to compactness since we are in a metric space. Let $\{m^{(k)}\}$ be a sequence in $\Delta(s,t)$. Since $\Delta(s,t)$ is tight, Prohorov's Theorem implies that there exists a subsequence $\{m^{(k_\ell)}\}$ that converges to some $m\in D$. We must show that $m\in\Delta(s,t)$. Since $m^{(k_\ell)}\in\Delta(s,t)$, there exists a $q^{(k_\ell)}\in D$ with $P_{st}(q^{(k_\ell)})=m^{(k_\ell)}$. Note that $\{m^{(k_\ell)}\}$ is a Cauchy sequence. From \eqref{eq: bounded trans} it follows that $\{q^{(k_\ell)}\}$ is also a Cauchy sequence and since $\ell^1$ is a Banach space, there exists an $q\in\ell^1$ with $q^{(k_\ell)}\to q$. By continuity of $P_{st}$ it follows that $P_{st}(q)=m$. In light of the discussion just below \eqref{eq: bounded trans}, $q\in D$ and hence $m\in\Delta(s,t)$.
\end{proof}

\begin{proof}[Proof of Theorem \ref{thrm: exist countable}]
Let $V$ be as in Lemma \ref{lemma: delta compact}. The proof is based on Theorem \ref{thrm: gen}. Let $V_n=\ell^1$ for $n\in\mathbb Z_-$ be the metric spaces. Let $t_1>t_2>\dots$ be the elements of $V$ in decreasing order and let $D(n)=\Delta(t_{n+1},t_n)$ for $n\in\mathbb Z_-$. Note that each $D(n)\subset V_n$ is a compact set by Lemma \ref{lemma: delta compact}. The continuous transformation associated with $V_n$ will be $P_{t_{n}t_{n-1}}$, $n<0$. All of these objects satisfy the required properties and we can use Theorem \ref{thrm: gen} to show that there exists a sequence of vectors $(m_{t_n})_{n\in\mathbb Z_-}$ such that for each $n$, $m_{t_n}\in D$ and $m_{t_n}=P_{t_{n+1}t_n}(m_{t_{n+1}})$. For any $t$ with $t_{n+1}<t<t_n$, we can take $m_t = P_{t_{n+1}t}(m_{t_{n+1}})$. This gives the result.
\end{proof}

We now give two examples where Condition P does not hold.\\

\noindent\textbf{Example 1.} This is an example where Condition P does not hold and a solution to \eqref{eq: defn m count} does not exist. 
Consider a symmetric random walk on countable state space $\mathbb Z$, where at each time, with probability $.5$ we take one step in the positive direction and with probability $.5$ we take a step in the negative direction. In this case, letting $E=\{2n:n=0,1,\dots\}$ be the even numbers, for $n>0$ we have
\begin{eqnarray*}
p_{-n,0}({i,k}) &=& (.5)^{n}{n \choose .5(n+|i-k|)} I_{[|i-k|\le n]} I_{[n-|i-k|\in E]}.
\end{eqnarray*}
It is easy to see that this does not satisfy Condition P. For the sake of contradiction, assume that the required probability measures $m_n$ exist. Now applying the monotonicity of binomial coefficients (see e.g.\ \cite{Andrica:Andreescu:2009}) and the well-known bound $\sqrt{2\pi}n^{n+.5}e^{-n}\le n!\le \sqrt{2e\pi}n^{n+.5}e^{-n}$ (see \cite{Robbins:1955}) gives
\begin{eqnarray*}
p_{-2n,0}({i,k}) \le (.5)^{2n}{2n \choose n} \le \sqrt\frac{e}{2\pi} n^{-1/2}.
\end{eqnarray*}
Hence
\begin{eqnarray*}
m_0(k) &=&  \sum_{i\in\mathbb Z} m_{-2n}(i) p_{-2n,0}(i,k) \le  \sum_{i\in\mathbb Z} m_{-2n}(i)  \sqrt\frac{e}{2\pi} n^{-1/2}= \sqrt\frac{e}{2\pi} n^{-1/2}\to0.
\end{eqnarray*}
Thus, $m_0(k)=0$ for each $k$, which contradiction the assumption that this is a probability measure.\\

\noindent\textbf{Example 2.} This is an example where Condition P does not hold, but a solution to \eqref{eq: defn m count} nevertheless does exist. 
A simple, but general situation, where this always holds is when each $P_n$ is onto. In this case, the image of $D$ through $P_n$ is $D$. However, in light of Lemma \ref{lemma: delta compact} and the fact that $D$ is not compact, Condition P does not hold in this case. A simple concrete example is a walk on $\mathbb Z$, where, for some fixed integer $\ell$, 
$$
p_{n}(i,k) = I_{[k=i+\ell]}.
$$
If $\ell=0$, then every state is absorbing. If $\ell=1$, then we have a walk similar to the one in Example $1$, but with a different probability of moving in the positive direction.\\

\begin{lemma}\label{lemma: conv}
a) Let $\{m_t\}_{t\in\mathbb Z_-}$ be a sequence in $D$. If there exists an $m\in D$, an infinite stochastic matrix $P_t^*$, and a sequence $(s_n)\subset\mathbb Z$ with $m_{s_n}\to m$ and $P_{s_nt}\to P^*_t$, then
$$
\lim_{n\to\infty} P_{s_nt}(m_{s_n})\to P_t^*(m).
$$ 
b) If Condition U holds, then for each $t\le0$ and any sequence $(s_n)\subset\mathbb Z$, there exists a further subsequence $s_{n_k}$ and an infinite stochastic matrix $P_t^*$ such that $P_{s_nt}\to P^*_t$.
\end{lemma}

\begin{proof}
We begin with a). By Skorohod's Representation Theorem, there is a probability space $(\Omega,\mathcal F,P)$ and $\mathbb N$-valued random variables $X,X_1,X_2,\dots$ on this space such that $X$ has distribution $m$, $X_n$ has distribution $m_{s_n}$, and $X_n(\omega)\to X(\omega)$ for each $\omega\in\Omega$.  Let $\mathrm E$ be the expectation operator on this probability space. Since the random variables are $\mathbb N$-valued, there is a function $M(\omega)$ such that if $n\ge M(\omega)$, then $X_n(\omega)=X(\omega)$. It follows that $p_{s_nt}({X_n(\omega) ,k})\to p_{t}^*(X(\omega),k)$ for each $\omega\in\Omega$. Let $m'_{s_n}=P_{s_nt}(m_{s_n})$, $m'=P_t^*(m)$, and note that
\begin{eqnarray*}
\lim_{n\to\infty} m'_{s_n}(k) &=& \lim_{n\to\infty} \sum_{i=1}^\infty m_{s_n}(i) p_{s_nt}({i,k})= \lim_{n\to\infty} \rE p_{s_n,t}({X_n ,k}) \\
&=& \rE p_{t}^*(X,k)=  \sum_{i=1}^\infty m(i) p^*_{t}({i,k})= m'(k),
\end{eqnarray*}
where we interchange limit and expectation using dominated convergence and the fact that $p_{s_n,t}({X_n ,k}) \le1$.

We now turn to b). Let $V$ be as in Lemma \ref{lemma: delta compact}. First, fix $t\in V$ and consider the sequence of matrices $\{P_{s_nt}:s_n<t\}$. Lemma \ref{lemma: delta compact} and \eqref{eq: which N epsilon} imply that the first rows of these matrices form a tight sequence of probability measures. Thus, there is a $m_1\in D$ and a sequence $(s_n^{(1)})\subset(s_n)$ such that $p_{s^{(1)}_nt}(1,\star)\to m_1$. Similarly, there is a $m_2\in D$ and a further subsequence $(s_n^{(2)})\subset (s_n^{(1)})$ such that $p_{s^{(2)}_nt}(\ell,\star)\to m_\ell$ for $\ell=1,2$. Continuing in this manner, we can find a sequence $m_1,m_2,m_3,\dots\in D$ and a collection of nested sequences $(s_n^{(1)})\supset (s_n^{(2)})\supset (s_n^{(3)})\supset\cdots$ with $\lim_{n\to\infty} p_{s^{(k)}_nt}(\ell,\star)= m_\ell$ for $k=1,2,\dots$ and $\ell=1,2,\dots,k$. Now, set $s^*_n=s_n^{(n)}$ and let $P_t^*$ be the infinite stochastic matrix such that $p_t^*(\ell,\star)=m_\ell$, $\ell=1,2,\dots$. It follows that $P_{s^*_nt}\to P_t^*$. 

Now, assume that $t\in\mathbb Z_-$ is not an element of $V$. By Condition U, there exists a $t_0<t$ with $t_0\in V$. Thus there is a sequence $(s^*_n)\subset(s_n)$ and a matrix $P_{t_0}^*$ such that $P_{s^*_nt_0}\to P^*_{t_0}$. Noting that $P_{s^*_nt}= P_{t-1}\circ\cdots\circ P_{t_0+1}\circ P_{t_0}\circ P_{s^*_nt_0}$, taking $P_{t}^*= P_{t-1}\circ\cdots\circ P_{t_0+1} \circ P_{t_0}\circ P_{t_0}^*$, and applying a) gives the result.
\end{proof}

\begin{proof}[Proof of Theorem \ref{thrm: uniq count}]
Let $V$ be as in Lemma \ref{lemma: delta compact}. The proof is similar to that of Theorem \ref{thrm: delta is simplex}, with several changes. First, we now use Lemma \ref{lemma: conv} to guarantee the existence of a sequence $(s_n)\subset V$ and an infinite stochastic matrix $P_t^*$ with $P_{s_nt}\to P^*_t$. Second, now $\Delta(t)$ is the convex hull of the vectors $a_1,a_2,\dots$, where $a_i = P_t^*(e_i)$ and $e_i\in D$ is the vector with $e_i(i)=1$ and $e_i(j)=0$ for $j\ne i$. Third, we now use Lemma \ref{lemma: delta compact} to guarantee that $(m_{s_n})$ is tight and Lemma \ref{lemma: conv} to show that 
$$
m_t=\lim_{k\to\infty}  P_{s_{n_k}t} (m_{s_{n_k}}) = P_t^*(m)= \sum_{i=1}^\infty m(i) e_i P_t^*= \sum_{i=1}^\infty m(i) a_i.
$$
Finally, we use  Lemma \ref{lemma: delta compact} to guarantee that $\Delta(s,t)$ is compact for each $s<t$.
\end{proof}

\section{Conclusions}\label{sec: conc}

In this paper we reviewed the zero-one law for MCs and gave two rigorous and detailed proofs. These had been missing from the literature. Further, in the case where the MC is indexed by $\mathbb Z$ or $\mathbb Z_-$ we gave a version of this law for the entrance $\sigma$-algebra. In the corresponding discussion, we noted an interesting dichotomy in two commonly used definitions of a MC. Further, to better understand when MCs on $\mathbb Z$ and $\mathbb Z_-$ exist, we extended a classical result due to Kolmogorov (1936) \cite{Kolm1936}. 

We conclude this paper by discussing several open problems. First, it seems that one should be able to extend the zero-one law for MCs to countable Markov Decision Processes with ``tail'' functionals, i.e.\ those that can be represented as indicators of some events in the tail $\sigma$-algebra, see \cite{Son91a} for a discussion of such processes.  Second, it would be interesting to obtain necessary and sufficient conditions for the existence of countable MCs indexed by $\mathbb Z$. We only  obtained a sufficient condition. Finally, as far as we know, the problem of characterizing the events in the tail and entrance $\sigma$-algebras for countable MCs has not been solved. Results are only available in the finite case, see e.g.\  \cite{Cohn70}, \cite{Cohn:1974}, and \cite{Iosifescu:1979}.
 
\section{Historical Notes}\label{sec: hist}

Homogeneous MCs are among the most fundamental concepts in probability theory and one of the most widely used probabilistic tools. However, in applications with a changing environment, these models are inadequate and one must study \nhm MCs instead. Whether due to their importance in applications or because of intrinsic interest, many prominent mathematicians and probabilists have been attracted to the study of \nhm MCs.  Even a short list is impressive: starting from A.~Markov himself as early as 1910. Other early pioneers include S.~Bernstein, R.~Dobrushin, W.~Doeblin, E.B.~Dynkin, and of course A.~Kolmogorov. This early work was continued and extended in papers by authors such as O.O.~Aalen, D.~Blackwell, X.R.~Cao, H.~Cohn, J.L.~Doob, D.~Griffeath, J.~Hajnal, D.~Hartfiel, W.J.~Hopp, G.A.~Hunt, M.~Iosifesku, J.F.~Kingman, S.E.~Kuznetsov, V.~Maksimov, S.~Molchanov, L.~Saloff-Coste, E.~Seneta, S.R.~Varadhan, D.~Williams and many others. 

While at first glance \nhm MCs may appear to have very little structure, much structure has nevertheless been uncovered. In particular, there are a number of results that aim to understand the events in the tail $\sigma$-algebra and how a MC approaches these events, while making essentially no assumptions on the underlying sequence of transition matrices. The zero-one law presented in this paper, and especially the backwards formulation given in Lemma \ref{lemma: approx}, fits into the lineage of such results. 

In the remainder of this section, we give a short overview of the history of several results of this type, which, over time, have evolved into the so-called Decomposition-Separation (DS) Theorem. This theorem generalizes to the nonhomogeneous case the well-known decomposition of the state space of a homogeneous Markov chain into transient and recurrent classes and further into cyclical subclasses. In the nonhomogeneous case, the decomposition is not just of the state space, but of the space-time representation. The only assumption is that the number of states is bounded.

The decomposition part of the theorem is primarily due to the work of Blackwell and Cohn. Motivated by Kolmogorov (1936) \cite{Kolm1936}, Blackwell studied the properties of MCs in reverse time. In the seminar paper Blackwell (1945) \cite{Bla1945}, he gave a partition of the space-time representation of the state space. To describe this, it helps to introduce several definitions, although the terminology was developed later. If $(S_n)_{n\in\mathbb Z_-}$ is the sequence of state spaces, then a sequence $J=(J_{n})_{n\in\mathbb Z_-}$, with $J_{n}\subset S_{n}$ is called a {\it jet}. A tuple of jets $(J^{1},...,J^{c})$ is called a {\it partition }of $(S_{n})_{n\in\mathbb Z_-}$ if $(J_{n}^{1},...,J_{n}^{c})$ is a partition of $S_{n}$ for every $n$. Blackwell proved that there exists a partition $(T^{1},...,T^{c})$ of $(S_{n})_{n\in\mathbb Z_-}$ such that the trajectories of the MC will, with probability one, reach and eventually stay in one of the jets $T^{i}$, $i=1,...,c$. This result was extended, in the works of Cohn, see \cite{Cohn70}, \cite{Cohn:1974}.  Cohn reformulated Blackwell's results in the context of MCs in forward time and proved that the tail $\sigma$-algebra of any nonhomogeneous MC consists of a finite number of atomic (indecomposable) sets, each of them related with a jet $T^{k}$ of Blackwell's partition. He also simplified Blackwell's proofs. 

The separation part of the DS theorem was proved by Sonin, one of the authors of this paper, in a series of papers \cite{Son87}, \cite{Son91a}, \cite{Son91b}. Here it was shown that there exist partitions into jets having the additional property that the {\it expected number of transitions} of trajectories of any MC $(Z_{n})$ between jets is{\it \ finite} on the infinite time interval. This separation property was not obvious and its existence had not been noted previously. Surveys about the DS theorem and related results can be found in \cite{Son96} and \cite{Son08}. The DS theorem has found applications in several areas, including simulated annealing, consensus algorithms, and probabilistic automata, see e.g.\ \cite{CoFi99}, \cite{Chatterjee:Tracol:2012}, \cite{Bolouki:Malhame:2016}, \cite{Etesami:2019}, and the references therein.

\end{document}